\documentclass[12pt]{amsart}
\usepackage{amsfonts,amsmath,a4wide,latexsym,amssymb,amscd,amsthm,amsxtra,mathrsfs,amsrefs}

\usepackage[bookmarksnumbered, colorlinks, plainpages]{hyperref}
\hypersetup{colorlinks=true,linkcolor=red, anchorcolor=green, citecolor=cyan, urlcolor=red, filecolor=magenta, pdftoolbar=true}

\usepackage{upref}
\usepackage{txfonts}
\usepackage{graphicx}

\usepackage[width=2cm, height=2cm, left=2cm, right=2cm, top=2cm, bottom=1cm]{geometry}

\allowdisplaybreaks

\theoremstyle{plain}
\newtheorem{thm}{Theorem}[section]
\newtheorem{lem}[thm]{Lemma}
\newtheorem{cor}[thm]{Corollary}

\theoremstyle{definition}
\newtheorem{rem}[thm]{Remark}

\newtheorem{defi}[thm]{Definition}
\newtheorem{ex}[thm]{Example}

\numberwithin{equation}{section}

\def\loc{\operatorname{loc}}
\def\rad{\operatorname{rad}}
\def\supp{\operatorname{supp}}

\def\esup{\operatornamewithlimits{ess\,sup}}

\def\R{\mathbb R}

\def\ap{\approx}

\def\mf{\mathfrak M}

\def\rn{\R^n}
\def\a{\alpha}

\def\la{\lambda}

\def\vp{\varphi}

\def\M{\mathcal M}

\def\dn{\downarrow}

\def\ls{\lesssim}
\def\gs{\gtrsim}
\def\ve{\varepsilon}
\def\R{\mathbb R}

\def\M{\mathcal M}

\def\Bxr {{B(x,r)}}

\def\Lploc{L_p^{\rm loc}(\rn)}

\def\Lloc{L_1^{\rm loc}(\rn)}

\def\B{\operatorname{BMO}}

\def\WM{\operatorname{\mathcal {W  \! M}}}

\begin{document}

\baselineskip=17pt

\title[]{Weak-type estimates in Morrey spaces for maximal commutator and commutator of maximal function}

\author[A. Gogatishvili]{Amiran Gogatishvili}
\address{Institute of Mathematics \\
Academy of Sciences of the Czech Republic \\
\v Zitn\'a~25 \\
115~67 Praha~1, Czech Republic} \email{gogatish@math.cas.cz}

\author[R.Ch.Mustafayev]{Rza Mustafayev}
\address{Department of Mathematics \\ Faculty of Science and Arts \\ Kirikkale
University \\ 71450 Yahsihan, Kirikkale, Turkey}
\email{rzamustafayev@gmail.com}

\author[M. A\v{g}cayazi]{M\"{u}jdat A\v{g}cayazi}
\address{Department of Mathematics \\ Faculty of Science and Arts \\ Kirikkale
University \\ 71450 Yahsihan, Kirikkale, Turkey}
\email{mujdat87@gmail.com}

\thanks{The research of A. Gogatishvili was partly supported by the grants P201-13-14743S of the Grant Agency of the Czech Republic and RVO: 67985840, by Shota Rustaveli National Science Foundation grants no. 31/48  (Operators in some function spaces and their applications in Fourier Analysis) and no. DI/9/5-100/13 (Function spaces, weighted inequalities for integral operators and problems of summability of Fourier series). The research of the first and second authors was partly supported by the joint project between  Academy of Sciences of Czech Republic and The Scientific and Technological Research Council of Turkey}



 \begin{abstract}
    In this paper it is shown that the Hardy-Littlewood maximal operator $M$ is not bounded on Zygmund-Morrey space $\mathcal{M}_{L(\log L),\lambda}$, but $M$ is still bounded on $\mathcal{M}_{L(\log L),\lambda}$ for radially decreasing functions. The boundedness of the iterated maximal operator $M^2$ from  $\mathcal{M}_{L(\log L),\lambda}$ to weak Zygmund-Morrey space $\WM_{L(\log L),\lambda}$ is proved. The class of functions for which the maximal commutator $C_b$ is bounded from  $\mathcal{M}_{L(\log L),\lambda}$ to $\WM_{L(\log L),\lambda}$
    are characterized. It is proved that the commutator of the
    Hardy-Littlewood maximal operator $M$ with function $b \in \B(\rn)$ such that $b^- \in L_{\infty}(\rn)$ is bounded from $\mathcal{M}_{L(\log L),\la}$ to $\WM_{L(\log L),\lambda}$. New pointwise characterizations of $M_{\alpha} M$ by means of
    norm of Hardy-Littlewood maximal function in classical Morrey spaces are given.
 \end{abstract}

\maketitle

\section{Introduction }

Given a locally integrable function $f$ on $\rn$ and $0\leq \a <n$,
the fractional maximal function $M_{\a}f$ of $f$ is defined by
$$
M_{\a}f(x):=\sup_{Q \ni x}|Q|^{\frac{\a-n}{n}}\int_Q |f(y)|\,dy,
\qquad (x\in\rn),
$$
where the supremum is taken over all cubes $Q$ containing $x$. The
operator $M_{\a}:~f \rightarrow M_{\a}f$ is called the fractional
maximal operator. $M: = M_{0}$ is the classical Hardy-Littlewood maximal operator.

The study of maximal operators is one of the most important topics in harmonic analysis.
These significant non-linear operators, whose behavior are very informative in particular in differentiation theory,
provided the understanding and the inspiration for the development of the general class of singular and potential operators
(see, for instance, \cite{stein1970,guz1975,GR,tor1986, stein1993,graf2008,graf}).

Let $f\in\Lloc$. Then $f$ is said to be in $\B (\rn)$ if the
seminorm given by
\begin{equation*}
\|f\|_{*}:= \sup_Q\frac{1}{|Q|}\int_Q |f(y)-f_Q|dy
\end{equation*}
is finite.

\begin{defi}
    Given  a measurable function $b$ the maximal commutator is defined
    by
    \begin{equation*}
    C_b(f)(x) : = \sup_{Q \ni x}\frac{1}{|Q|}\int_{Q}
    |b(x)-b(y)||f(y)|dy
    \end{equation*}
    for all $x\in\rn$.
\end{defi}
This operator plays an important role in the study of commutators of
singular integral operators with $\B$ symbols (see, for instance, \cite{GHST,LiHuShi,ST1,ST2}). The maximal operator $C_b$ has been studied intensively and there
exist plenty of results about it. Garcia-Cuerva
et al. \cite{GHST} proved the following statement.
\begin{thm}\label{Cb}
    Let $1 < p < \infty$. $C_b$ is bounded on $L_p(\rn)$ if and only if
    $b\in\B(\rn)$.
\end{thm}

\begin{defi}
    Given a measurable function $b$ the commutator of the
    Hardy-Littlewood maximal operator $M$ and $b$ is defined by
    \begin{equation*}
    [M,b]f (x): = M(bf)(x) - b(x)Mf(x)
    \end{equation*}
    for all $x\in\rn$.
\end{defi}

The operator $[M,b]$ was studied by Milman et al. in \cite{MilSchon}
and \cite{BasMilRu}. This operator arises, for example, when one
tries to give a meaning to the product of a function in $H^1$ and a
function in $\B$ (which may not be a locally integrable function,
see, for instance, \cite{bijz}). Using real interpolation
techniques, in \cite{MilSchon}, Milman and Schonbek proved the $L_p$-boundedness of the operator $[M,b]$. Bastero, Milman and Ruiz \cite{BasMilRu} proved the next theorem.
\begin{thm}\label{BMR}
    Let $1 < p < \infty$. Then the following assertions are equivalent:

    {\rm (i)} $[M,b]$ is bounded on $L_p(\rn)$.

    {\rm (ii)} $b \in \B(\rn)$ and $b^- \in L_{\infty}(\rn)$.
    \footnote{Denote by $b^+(x)=\max\{b(x),0\}$ and
        $b^-(x)=-\min\{b(x),0\}$, consequently $b=b^+-b^-$ and
        $|b|=b^++b^-$.}
\end{thm}

The opertors $C_b$ and $[M,b]$ enjoy weak-type $L(1 + \log^+ L)$ estimate.
\begin{thm}[{\cite[Theorem 1.5]{AGKM}, see also \cite{HuLinYang} and \cite{HuYang}}] \label{HuYang} The following assertions are equivalent:

    {\rm (i)} There exists a positive constant $c$ such that for each $\la >0$, inequality
    \begin{equation}\label{eq0003.7}
    |\{x \in\rn :  C_b(f)(x)> \la \}| \le c
    \int_{\rn}\frac{|f(x)|}{\la}\left(1+\log^+
    \left(\frac{|f(x)|}{\la}\right)\right)dx.
    \end{equation}
    holds for all $f\in L(1 + \log^+ L)(\rn)$.

    {\rm (ii)} $b\in \B (\rn)$.
\end{thm}

\begin{thm} \cite[Theorem 1.6]{AGKM} \label{thm3495195187t}
    Let $b \in \B(\rn)$ such that $b^-\in L_{\infty}(\rn)$. Then there
    exists a positive constant $c$ such that
    \begin{align}
    |\{x \in\rn : |[M,b] f(x)|  > \la \}| \leq c c_0\left(1+\log^+
    c_0\right)\int_{\rn}\frac{|f(x)|}{\la}\left(1+\log^+\left(\frac{|f(x)|}{\la}\right)\right)dx, \label{weak}
    \end{align}
    for all $f\in L\left(1 + \log^+ L\right)$ and $\la >0$, where
    $c_0=\|b^+\|_{*}+\|b^-\|_{\infty}$.
\end{thm}

Operators $C_b$ and $[M,b]$ essentially differ from each other. For
example, $C_b$ is a positive and sublinear operator, but $[M,b]$ is
neither positive nor sublinear. However, if $b$ satisfies some
additional conditions, then operator $C_b$ controls $[M,b]$.
\begin{lem}\cite[Lemma 3.1 and 3.2]{AGKM}\label{pointwise1}
    Let $b$ be any non-negative locally integrable function. Then
    \begin{equation}\label{eqPointwise}
    |[M,b]f(x)|\leq C_b(f)(x) \quad (x \in \rn)
    \end{equation}
    holds for all $f \in L_1^{\loc}(\rn)$.

    If $b$ is any locally integrable function on $\rn$, then
    \begin{equation}\label{eq.0001}
    |[M,b]f|(x)\leq C_b(f)(x)+ 2b^-(x) Mf(x) \quad (x \in \rn)
    \end{equation}
    holds for all $f \in L_1^{\loc}(\rn)$.
\end{lem}

We recall the following statement from \cite{AGKM}.
    \begin{thm}\cite[Theorem 1.13]{AGKM}\label{thm1}
        Let $b \in \B(\rn)$. Suppose that $X$ is a Banach space of
        measurable functions defined on $\rn$. Moreover, assume that $X$
        satisfies the lattice property, that is,
        $$
        0 \le g \le f \quad \Rightarrow \quad \|g\|_X \ls \|f\|_X.
        $$
        Assume that $M$ is bounded on $X$. Then the operator $C_b$ is
        bounded on $X$, and the inequality
        $$
        \|C_b f\|_X \le c \|b\|_* \|f\|_X
        $$
        holds with constant $c$ independent of $f$.

        Moreover, if $b^- \in L_{\infty}(\rn)$, then the operator $[M,b]$ is
        bounded on $X$, and the inequality
        $$
        \|[M,b] f\|_X \le c (\|b^+\|_* + \|b^-\|_{\infty})\|f\|_X
        $$
        holds with constant $c$ independent of $f$.
    \end{thm}

    The proof of previous theorem is based on the following inequalities.
    \begin{thm}\cite[Corollary 1.11 and 1.12]{AGKM}\label{lem1111111}
        Let $b\in \B(\rn)$. Then, there exists a positive constant $c$ such
        that
        \begin{equation}\label{eq.0002}
        C_b(f)(x)\leq c \|b\|_{*} M^2f(x) \qquad (x\in\rn)
        \end{equation}
        for all $f \in L_1^{\loc}(\rn)$.

        Moreover, if $b^- \in L_{\infty}(\rn)$, then, there
        exists a positive constant $c$ such that
        \begin{equation}\label{eqPointwise3}
        |[M,b]f(x)|\leq c \left(\|b^+\|_{*}+\|b^-\|_{\infty}\right)M^2f(x)
        \end{equation}
        for all $f \in  L_1^{\loc}(\rn)$.
    \end{thm}

    The classical Morrey spaces $\mathcal{M}_{p, \lambda} \equiv \mathcal{M}_{p, \lambda} (\rn)$,
    were  introduced by C.~Morrey in \cite{M1938} in order to study regularity questions which appear in the Calculus of Variations,
    and defined as follows:  for $0 \le \lambda \le n$ and $1\le p \le \infty$,
    $$
    \mathcal{M}_{p,\lambda} : = \left\{ f \in   \Lploc:\,\left\| f\right\|_{\mathcal{M}_{p,\lambda }} : =
    \sup_{x\in \rn, \; r>0 }
    r^{\frac{\lambda-n}{p}} \|f\|_{L_{p}(B(x,r))} <\infty\right\},
    $$
    where $\Bxr$ is the open ball centered at $x$ of radius $r$.

    Note that $\mathcal{M}_{p,0}(\rn) = L_{\infty}(\rn)$ and ${\mathcal M}_{p,n}(\rn) = L_{p}(\rn)$.

    These spaces describe local regularity more precisely than Lebesgue spaces and appeared to be quite useful in the study of the local
    behavior of solutions to partial differential equations, a priori estimates and other topics in PDE (cf. \cite{giltrud}).

    The boundedness of the Hardy-Littlewood maximal operator $M$ in
    Morrey spaces $\mathcal{M}_{p,\lambda}$ was proved by F.~Chiarenza
    and M.~Frasca in \cite{ChiFra1987}: It was shown  that $Mf$ is a.e. finite if $f \in \mathcal{M}_{p,\lambda}$ and an estimate
    \begin{equation}\label{ChiFr}
    \|Mf\|_{\mathcal{M}_{p,\lambda}} \le c \|f\|_{\mathcal{M}_{p,\lambda}}
    \end{equation}
    holds if $1 < p < \infty$ and $0 < \lambda < n$, and a weak type estimate \eqref{ChiFr} replaces for $p = 1$, that is, the inequality
    \begin{equation}\label{ChiFrWeakType}
    t |\{Mf > t\} \cap B(x,r)| \le cr^{n - \lambda}
    \|f\|_{\mathcal{M}_{1,\lambda}}
    \end{equation}
    holds with constant $c$ independent of $x,\,r,\,t$ and $f$.

    In \cite{gogmus}, it is proved that the Hardy-Littlewood maximal operator $M$ is bounded on $\M_{1,\la}$, $0 \le \la < n$,
    for radially decreasing functions, that is, the inequality
    \begin{equation}\label{gogmus}
    \|Mf\|_{\M_{1,\la}} \ls \|f\|_{\M_{1,\la}},~f \in \mf^{\rad,\dn}
    \end{equation}
    holds with constant independent of $f$, and an example which shows that $M$ is not bounded on $\M_{1,\la}$, $0 < \la < n$ is given.

    Combining Theorem \ref{lem1111111} with inequalities \eqref{ChiFr} and \eqref{gogmus}, it is easy to generalize Theorems \ref{Cb} and \ref{BMR} to Morrey spaces (see Theorems \ref{thm3.1} and \ref{thm4.5}).

    In this paper the Zygmund-Morrey and the weak Zygmund-Morrey spaces are
    defined. In order to investigate the boundedness of the maximal commutator $C_b$ and the commutator of maximal function $[M,b]$ on Zygmund-Morrey spaces we start to study the boundedness properties of the Hardy-Littlewood maximal operator on these spaces. It is shown that the Hardy-Littlewood maximal operator $M$ is not bounded on Zygmund-Morrey spaces $\mathcal{M}_{L(\log L),\la}$, but $M$ is still bounded on $\mathcal{M}_{L(\log L),\la}$ for radially decreasing functions. The boundedness of the iterated maximal operator $M^2$ from Zygmund-Morrey spaces $\mathcal{M}_{L(\log L),\la}$ to weak Zygmund-Morrey spaces $\WM_{L(\log L),\la}$ is proved. The class of functions for which the maximal commutator $C_b$ is bounded from  $\mathcal{M}_{L(\log L),\la}$ to $\WM_{L(\log L),\la}$
    are characterized. It is proved that the commutator $[M,b]$ is bounded from $\mathcal{M}_{L(\log L),\la}$ to $\WM_{L(\log L),\la}$, when $b \in \B(\rn)$ such that $b^- \in L_{\infty}(\rn)$. New pointwise characterizations of $M_{\alpha} M$ by means of
    norm of Hardy-Littlewood maximal function in Morrey space are given.

    The paper is organized as follows. In Section \ref{sect2} notations and preliminary results are given. Boundedness of maximal commutator
    and commutator of maximal function in Morrey spaces are
    investigated in Section \ref{sect6}. New characterizations of $M_{\a}M$
    are obtained in section \ref{sect7}. In Section \ref{sect7.3} it is shown that the Hardy-Littlewood maximal operator $M$ is not bounded on Zygmund-Morrey spaces $\mathcal{M}_{L(\log L),\la}$, but $M$ is still bounded on $\mathcal{M}_{L(\log L),\la}$ for radially decreasing functions. The boundedness of the iterated maximal operator  from $\mathcal{M}_{L(\log L),\la}$ to $\WM_{L(\log L),\la}$ is proved in Section \ref{sect7.5}. In Section \ref{sect8} weak-type estimates for
    maximal commutator and commutator of maximal function in Zygmund-Morrey spaces are proved.


    \section{Notations and Preliminaries}\label{sect2}

    Now we make some conventions. Throughout the paper, we always denote
    by $c$ a positive constant, which is independent of main
    parameters, but it may vary from line to line. However a constant
    with subscript such as $c_1$ does not change in different
    occurrences. By $a\ls b$ we mean that $a\le c b$ with some positive
    constant $c$ independent of appropriate quantities. If $a\ls b$ and
    $b\ls a$, we write $a\approx b$ and say that $a$ and $b$ are
    equivalent. For a measurable set $E$, $\chi_E$ denotes the
    characteristic function of $E$. Throughout this paper cubes will be
    assumed to have their sides parallel to the coordinate axes. Given
    $\la > 0$ and a cube $Q$, $\la Q$ denotes the cube with the same
    center as $Q$ and whose side is $\la$ times that of $Q$. For a fixed
    $p$ with $p\in [1,\infty)$, $p'$ denotes the dual exponent of $p$,
    namely, $p'=p/(p-1)$. For any measurable set $E$ and any integrable
    function $f$ on $E$, we denote by $f_E$ the mean value of $f$ over
    $E$, that is, $f_E=(1/|E|)\int_E f(x)dx$. Unless a special remark is made, the
    differential element $dx$ is omitted when the integrals under
    consideration are the Lebesgue integrals.

    For the sake of completeness we recall the definitions and some
    properties of the spaces we are going to use.

    Let $\Omega$ be any measurable subset of $\rn$, $n\geq 1$. Let
    $\mf(\Omega)$ denote the set of all measurable functions on $\Omega$
    and $\mf_0 (\Omega)$ the class of functions in $\mf (\Omega)$ that
    are finite a.e.

    For $p\in (0,\infty]$, we define the functional
    $\|\cdot\|_{p,\Omega}$ on $\mf(\Omega)$ by

    \begin{equation*}
    \|f\|_{p,\Omega}:=
    \begin{cases}
    (\int_{\Omega} |f(x)|^p \,dx)^{1/p}  &\text{if } \ \ p<\infty,\\
    \esup_{\Omega} |f(x)|              &\text{if } \ \ p=\infty.
    \end{cases}
    \end{equation*}

    The Lebesgue
    space $L_p(\Omega)$ is given by
    \begin{equation*}
    L_p(\Omega):= \{f\in \mf(\Omega): \|f\|_{p,\Omega}<\infty\}
    \end{equation*}
    and it is equipped with the quasi-norm $\|\cdot\|_{p,\Omega}$.

    Denote by $\mf^{\rad,\dn} = \mf^{\rad,\dn}(\rn)$
    the set of all measurable, radially decreasing functions on
    $\rn$, that is,
    $$
    \mf^{\rad,\dn} : = \{f \in \mf(\rn):\, f(x) = \vp(|x|),\,x \in \rn
    \,\mbox{with}\,\vp \in \mf^{\dn}(0,\infty)\}.
    $$

    Recall that $Mf \ap Hf$, $f \in \mf^{\rad,\dn}$, where
    $$
    Hf (x) : = \frac{1}{|B(0,|x|)|} \int_{B(0,|x|)} |f(y)|\,dy
    $$
    is $n$-dimensional Hardy operator. Obviously, $Hf \in  \mf^{\rad,\dn}$, when $f \in \mf^{\rad,\dn}$.

    The non-increasing rearrangement (see, e.g., \cite[p. 39]{BS}) of a
    function $f \in \mf_0 (\rn)$  is defined by
    $$
    f^*(t) : =\inf\left\{\la >0 : |\{x\in\rn: |f(x)|>\la \}| \leq t
    \right\}\quad (0<t<\infty).
    $$
    Then $f^{**}$ will denote the maximal function of $f^*$ defined by
    $$
    f^{**} (t) : = \frac{1}{t} \int_0^t f^* (s)\,ds, ~ (t > 0).
    $$

    The Zygmund class $L(\log^+ L)(\Omega)$ is the set of all $f \in
    \mf(\Omega)$ such that
    $$
    \int_{\Omega} |f(x)| (\log^+ |f(x)|)\,dx < \infty,
    $$
    where $\log^+ t = \max \{\log t,0\}$, $t>0$. Generally, this is not
    a linear set. Nevertheless, considering the class
    $$
    L(1 + \log^+ L)(\Omega) = \left\{ f \in \mf (\Omega):\, \|f\|_{L(1 +
        \log^+ L)(\Omega)} : = \int_{\Omega} |f(x)|\, (1 + \log^+
    |f(x)|)\,dx < \infty \right\},
    $$
    we obtain a linear set, the Zygmund space.

    The size of $M^2$ is given by the following inequality.
    \begin{lem}\cite[Lemma 1.6]{CPer}\label{lem002.8}
        There exists a positive constant $c$ such that for any function $f$ and for all $\la >0$,
        \begin{equation}
        |\{x  \in\rn :  M^2f(x)> \la \}|  \leq c
        \int_{\rn}\frac{|f(x)|}{\la}\left(1+\log^+
        \left(\frac{|f(x)|}{\la}\right)\right)dx.
        \end{equation}
    \end{lem}

    The following important result regarding $\B$ is true.

    \begin{lem}[{\cite{JN} and \cite{BenSharp}}]\label{lem2.4.}
        For $p\in (0,\infty)$, $\B(p)(\rn)=\B(\rn)$, with equivalent norms,
        where
        $$
        \|f\|_{\B(p)(\rn)}:=\sup_{Q}\left(
        \frac{1}{|Q|}\int_{Q}|f(y)-f_{Q}|^p dy\right)^{\frac{1}{p}}.
        $$
    \end{lem}

    A continuously increasing function
    on $[0,\infty]$, say $\Psi: [0,\infty]\rightarrow [0,\infty]$ such
    that $\Psi (0)=0$, $\Psi(1)=1$ and $\Psi(\infty)=\infty$, will be
    referred to as an Orlicz function. If $\Psi$ is an Orlicz function,
    then
    $$
    \Phi (t)=\sup \{ts-\Psi(s); s\in [0,\infty]\}
    $$
    is the complementary Orlicz function to $\Psi$.

    The Orlicz space denoted by $L^{\Psi} = L^{\Psi}(\rn)$ consists
    of all measurable functions $g: \rn \rightarrow \R$ such that
    $$
    \int_{\rn} \Psi \left(\frac{|g(x)|}{\a}\right)dx <\infty
    $$
    for some $\a >0$.

    Let us define the $\Psi$-average of $g$ over a cube $Q$ of $\rn$ by
    \begin{equation*}
    \|g\|_{\Psi,Q}=\inf \left\{\a>0: \frac{1}{|Q|}\int_{Q}
    \Psi\left(\frac{|g(x)|}{\a}\right)dx\leq 1\right\}.
    \end{equation*}

    When $\Psi$ is a Young function, i.e. a convex Orlicz function, the
    quantity
    \begin{equation*}
    \|f\|_{\Psi}=\inf \left\{\a>0: \int_{\rn}
    \Psi\left(\frac{|f(y)|}{\a}\right)dy\leq 1\right\}
    \end{equation*}
    is well known Luxemburg norm in the space $L^{\Psi}$ (see \cite{RR}).

    A Young function $\Psi$ is said to satisfy the $\nabla_2$-condition, denoted $\Psi \in \nabla_2$, if for some $K > 1$
    $$
    \Psi (t) \le \frac{1}{2K} \Psi (Kt) ~\mbox{for all} ~ t > 0.
    $$
    It should be noted that $\Psi (t) \equiv t$ fails the $\nabla_2$-condition.

    \begin{thm}\cite{k}\label{k}
        The Hardy-Littlewood maximal operator is bounded on $L^{\Psi}$, provided that $\Psi \in \nabla_2$.
    \end{thm}
    Combining Theorem \ref{k} and \ref{thm1}, we  obtain the following statement.
    \begin{thm}
        Let $b \in \B(\rn)$ and $\Psi \in \nabla_2$.

        Then the operator $C_b$ is bounded on $L^{\Psi}$, and the inequality
        $$
        \|C_b f\|_{L^{\Psi}} \le c \|b\|_* \|f\|_{L^{\Psi}}
        $$
        holds with constant $c$ independent of $f$.

        Moreover, if $b^- \in L_{\infty}(\rn)$, then the operator $[M,b]$ is
        bounded on $L^{\Psi}$, and the inequality
        $$
        \|[M,b] f\|_{L^{\Psi}} \le c (\|b^+\|_* + \|b^-\|_{\infty})\|f\|_{L^{\Psi}}
        $$
        holds with constant $c$ independent of $f$.
    \end{thm}

    If $f\in L^{\Psi}(\rn)$, the Orlicz maximal function of $f$ with
    respect to $\Psi$ is defined by setting
    \begin{equation*}
    M_{\Psi}f(x)=\sup_{x\in Q} \|f\|_{\Psi,Q},
    \end{equation*}
    where the supremum is taken over all cubes $Q$ of $\rn$ containing $x$.

    The generalized H\"older's inequality
    \begin{equation}\label{genHolder}
    \frac{1}{|Q|}\int_{Q} |f(y)g(y)|dy \leq \|f\|_{\Phi,Q}
    \|g\|_{\Psi,Q},
    \end{equation}
    where $\Psi$ is the complementary Young function associated to
    $\Phi$, holds.

    The main example that we are going to be using is
    $\Phi(t)=t(1+\log^+ t)$ with maximal function defined by $M_{L(1 +
        \log^+ L)}$. The complementary Young function is given by
    $\Psi(t)\thickapprox e^t$ with the corresponding maximal function
    denoted by $M_{\exp L}$.

    We define the weak  $L(1 + \log^+ L)$-average of $g$ over a cube $Q$
    of $\rn$ analogously by
    \begin{equation*}
    \|g\|_{WL(1 + \log^+ L),Q}=\inf \left\{\a>0:
    \sup_{t>0}\frac{1}{|Q|}\frac{|\{x\in Q: |g(x)|>\a
        t\}|}{\frac{1}{t}\left(1+\log^+ \frac{1}{t}\right)}\leq 1\right\}.
    \end{equation*}

    Let $0<\la <n$. The Zygmund-Morrey spaces $\mathcal{M}_{L(\log L),\la}(\rn) \equiv \mathcal{M}_{L(1 + \log^+ L),\la}(\rn)$ and the weak Zygmund-Morrey spaces $\WM_{L(\log L),\la}(\rn) \equiv \WM_{L(1 + \log^+ L),\la}(\rn)$ are defined as follows:
    \begin{align*}
    \mathcal{M}_{L(1 + \log^+ L),\la}(\rn): = & \{ f \in \mf
    (\rn):\,\|f\|_{\mathcal{M}_{L(1+\log^+
            L),\la}}:=\sup_{Q}|Q|^{\frac{\la}{n}}\|f\|_{L(1 + \log^+ L),Q}<\infty \}, \\
    \WM_{L(1 + \log^+ L),\la}(\rn)  : = &  \{f \in \mf
    (\rn):\,\|f\|_{\WM_{L(1+\log^+
            L),\la}}:=\sup_{Q}|Q|^{\frac{\la}{n}}\|f\|_{WL(1 + \log^+ L),Q}<\infty \},
    \end{align*}
    respectively. Note that $\mathcal{M}_{L(1+\log^+
        L),\la}$ is a special case of Orlicz-Morrey spaces ${\mathcal L}^{\Phi,\phi}$ (with  $\Phi (t) = t (1 + \log^+ t)$ and $\phi (t) = t^{\lambda}$, $t > 0$) defined in \cite[Definitions 2.3]{sst}. As we know, a weak version has not been defined yet in such form.

    \section{Boundedness of maximal commutator and commutator of maximal function in Morrey spaces}\label{sect6}

    In this section we investigate boundedness of maximal commutator and
    commutator of maximal function in Morrey spaces.

    The following theorem is true.
    \begin{thm}\label{thm3.1}
        Let $1<p<\infty$, $0\leq\la\leq n$. The following assertions are
        equivalent:

        {\rm (i)} $b\in \B(\rn)$.

        {\rm (ii)} The operator $C_b$ is bounded on $\mathcal{M}_{p,\lambda
        }$.

    \end{thm}
    \begin{proof}
        {${\rm (i)}\Rightarrow {\rm (ii)}$}. Suppose that $b\in\B(\rn)$. By
        Theorem \ref{lem1111111} and inequality \eqref{ChiFr} it follows that
        $C_b$ is bounded in Morrey space $\mathcal{M}_{p,\lambda }$ and the
        following inequality holds:
        $$
        \|C_b(f)\|_{\mathcal{M}_{p,\lambda }}\lesssim \|b\|_{*}\,
        \|f\|_{\mathcal{M}_{p,\lambda }}.
        $$

        {${\rm (ii)}\Rightarrow {\rm (i)}$}. Assume that there exists $c > 0$ such that
        $$
        \|C_b(f)\|_{\mathcal{M}_{p,\lambda}} \le c \|f\|_{\mathcal{M}_{p,\lambda}}
        $$
        for all $f \in \mathcal{M}_{p,\lambda}$.
        Obviously,
        $$
        \|f\|_{\mathcal{M}_{p,\lambda }}\thickapprox \sup_{Q'}
        \left(|Q'|^{\frac{\la-n}{n}}\int_{Q'}|f(y)|^pdy\right)^{\frac{1}{p}}.
        $$
        Let $Q$ be a fixed cube. We consider $f=\chi_Q$. It is easy to
        compute that
        \begin{equation}\label{eq0000005.1}
        \begin{split}
        \|\chi_Q\|_{\mathcal{M}_{p,\lambda }}&\thickapprox \sup_{Q'}
        \left(|Q'|^{\frac{\la-n}{n}}\int_{Q'}\chi_Q(y)dy\right)^{\frac{1}{p}}=\sup_{Q'}\left(|Q'\cap
        Q||Q'|^{\frac{\la-n}{n}} \right)^{\frac{1}{p}} \\
        & =\sup_{Q'\subseteq Q}\left(|Q'||Q'|^{\frac{\la-n}{n}}
        \right)^{\frac{1}{p}}=|Q|^{\frac{\la}{np}}.
        \end{split}
        \end{equation}

        On the other hand, since
        $$
        C_b(\chi_Q)(x)\gtrsim \frac{1}{|Q|}\int_{Q}|b(y)-b_Q|dy \quad
        \mbox{for all} \quad x\in Q.
        $$
        then
        \begin{equation}\label{eq0000005.2}
        \begin{split}
        \|C_b(\chi_Q)\|_{\mathcal{M}_{p,\lambda }} & \thickapprox \sup_{Q'}
        \left(|Q'|^{\frac{\la-n}{n}}\int_{Q'}|C_b(\chi_Q)(y)|^pdy\right)^{\frac{1}{p}}
        \\
        &\gtrsim |Q|^{\frac{\la}{np}}\frac{1}{|Q|}\int_{Q}|b(y)-b_Q|dy.
        \end{split}
        \end{equation}
        Since by assumption
        $$
        \|C_b(\chi_Q)\|_{\mathcal{M}_{p,\lambda }}\lesssim
        \|\chi_Q\|_{\mathcal{M}_{p,\lambda }},
        $$
        by \eqref{eq0000005.1} and \eqref{eq0000005.2}, we get that
        $$
        \frac{1}{|Q|}\int_{Q}|b(y)-b_Q|dy \lesssim c.
        $$
    \end{proof}

    Combining Theorem \ref{lem1111111} with inequality \eqref{gogmus}, we get the following statement.
    \begin{thm}\label{thm3.11}
        Let $0 < \la < n$. Assume that $b\in \B(\rn)$. Then the operator $C_b$ is bounded on $\mathcal{M}_{1,\lambda}$ for radially decreasing functions.
    \end{thm}

    The following theorem was proved in \cite{x}.
    \begin{thm}\label{thm4.5}
        Let $1<p<\infty$, $0\le\la \le n$. Suppose that $b$ be a real
        valued, locally integrable function in $\rn$. The following
        assertions are equivalent:

        {\rm (i)} $b$ is in $\B(\rn)$ such that $b^-\in L_{\infty}(\rn)$.

        {\rm (ii)} The commutator $[M,b]$ is bounded in
        $\mathcal{M}_{p,\lambda }$.
    \end{thm}

    \begin{rem}
        {${\rm (i)}\Rightarrow {\rm (ii)}$}. Assume that $b$ is in $\B(\rn)$
        such that $b^-\in L_{\infty}(\rn)$. By Theorem \ref{lem1111111} and
        inequality \eqref{ChiFr} it follows that $[M,b]$ is bounded in Morrey
        space $\mathcal{M}_{p,\lambda }$ and the following inequality holds:
        $$
        \|[M,b]f\|_{\mathcal{M}_{p,\lambda }}\lesssim
        \left(\|b^+\|_{*}+\|b^-\|_{\infty}\right)\,
        \|f\|_{\mathcal{M}_{p,\lambda }}.
        $$
    \end{rem}

    Combining Theorem \ref{lem1111111} with inequality \eqref{gogmus}, we obtain the following statement.
    \begin{thm}\label{thm4.51}
        Let $0 < \la < n$. Suppose that $b$ is in $\B(\rn)$ such that $b^-\in L_{\infty}(\rn)$. Then  $[M,b]$ is bounded on $\mathcal{M}_{1,\lambda}$ for radially decreasing functions.
    \end{thm}

    \section{Some auxillary results}\label{sect7}

    To prove the theorems in the next sections we need the following results.
    \begin{thm}\label{thm6.3}
        Let $0\leq \a < n$. Then
        \begin{equation*}
        \begin{split}
        M_{\a}(Mf)(x) = \sup_{Q \ni x}|Q|^{\frac{\a-n}{n}}\int_{Q}Mf &
        \thickapprox
        \sup_{Q \ni x}|Q|^{\frac{\a}{n}}\|f\|_{L(1 + \log^+ L),Q} \\
        &\thickapprox \sup_{Q \ni x}|Q|^{\frac{\a-n}{n}}\int_{Q}|f|\left(1+\log^+
        \frac{|f|}{|f|_{Q}}\right)
        \end{split}
        \end{equation*}
        holds for all $f\in\Lloc$.
    \end{thm}
    The statement of Theorem \ref{thm6.3} follows by the following lemmas.

\begin{lem}\label{lem6.1}
	The inequality
	\begin{equation}
	\frac{1}{|Q|}\int_{Q}Mf(y)dy \ls \sup_{Q\subset Q'}\|f\|_{L(1 +
		\log^+ L),Q'}
	\end{equation}
	holds for all $f\in \Lloc$ with positive constant independent of
	$f$ and $Q$.
\end{lem}
\begin{proof}
	Let $Q$ be a cube in $\rn$ and $f=f_1+f_2$, where $f_1=f\chi_{3Q}$.
	Then
	\begin{equation}\label{eq434565}
	\frac{1}{|Q|}\int_{Q}Mf(y)dy \leq
	\frac{1}{|Q|}\int_{Q}Mf_1(y)dy+\frac{1}{|Q|}\int_{Q}Mf_2(y)dy.
	\end{equation}
	We recall simple geometric observation: for a fixed point $x \in Q$, if a cube $Q'$ satisfies $Q' \ni x$ and $Q' \cap (3Q)^c \neq \emptyset$, then $Q \subset 3 Q'$. Hence
	$$
	Mf_2 (x) = \sup_{Q' \ni x} \frac{1}{|Q'|}\int_{Q'}|f_2(y)|dy \le \sup_{Q\subset 3Q'}\frac{1}{|Q'|}\int_{Q'}|f(y)|dy.
	$$
	Consequently, we have that
	\begin{equation}\label{eq434596}
	\frac{1}{|Q|}\int_{Q}Mf_2(y)dy \lesssim \sup_{Q\subset
		Q'}\frac{1}{|Q'|}\int_{Q'}|f(y)|dy.
	\end{equation}
	Since for any cube $Q'$
	$$
	\frac{1}{|Q'|}\int_{Q'}|f(y)|dy\leq \|f\|_{L(1 + \log^+ L),Q'},
	$$
	we get
	\begin{equation}\label{eq996677}
	\frac{1}{|Q|}\int_{Q}Mf_2(y)dy \lesssim \sup_{Q\subset Q'}\|f\|_{L(1
		+ \log^+ L),Q'}.
	\end{equation}
	
	On the other hand
	$$
	\frac{1}{|Q|}\int_{Q}Mf(y)dy \ls \|f\|_{L(1 + \log^+ L),Q}
	$$
	for all $f$ such that $\supp f \subset Q$ (see \cite[p. 174]{CPer}).
	Thus
	\begin{equation}\label{eq996633}
	\frac{1}{|Q|}\int_{Q}Mf_1(y)dy\lesssim
	\frac{1}{|3Q|}\int_{3Q}Mf_1(y)dy \lesssim \|f\|_{L(1 + \log^+
		L),3Q}.
	\end{equation}
	
	From \eqref{eq434565}, \eqref{eq996677} and \eqref{eq996633}, it follows
	that
	\begin{equation}\label{eq44005500}
	\begin{split}
	\!\!\frac{1}{|Q|}\int_{Q}Mf(y)dy \lesssim \sup_{Q\subset
		Q'}\|f\|_{L(1 + \log^+ L),Q'} + \|f\|_{L(1 + \log^+ L),3Q} \lesssim \sup_{Q\subset Q'}\|f\|_{L(1 + \log^+ L),Q'}.
	\end{split}
	\end{equation}
\end{proof}

We recall the following statement (see, for instance, \cite[p. 175]{IwMar}). For the completeness we give the proof.  
\begin{lem}\label{stein}
Note that the estimation
$$
\int_Q M(f \chi_Q) \ap \int_Q |f| \,\bigg( 1 + \log^+ \frac{|f|}{|f|_Q}\bigg),
$$
holds for all $f \in L_1^{\loc}(\rn)$ with positive constants independent of $f$ and $Q$.
\end{lem}
    \begin{proof}
    	Let $Q$ be a cube in $\rn$. We are going to use weak type estimates
    	(see \cite{stein1969}, for instance): there exist positive constants
    	$c_1 <1$  and $c_2 > 1$ such that for every $f\in \Lloc$ and for
    	every $t>1/|Q|\int_{Q}|f|$ we have
    	$$
    	c_1 \int_{\{x\in Q: |f(x)|>t\}}\frac{|f(x)|}{t}\,dx \le |\{x\in Q:
    	M(f\chi_Q)(x)>t\}| \le c_2  \int_{\{x\in Q:
    		|f(x)|>t/2\}}\frac{|f(x)|}{t}\,dx.
    	$$
    	We have that
    	\begin{align*}
    	\int_{Q}M(f\chi_Q) & = \int_0^{\infty}|\{x\in Q: M(f\chi_Q)(x)>\la\}|d\la \\
    	& = \int_0^{|f|_Q}|\{x\in Q: M(f\chi_Q)(x)>\la\}|d\la \\
    	&+ \int_{|f|_Q}^{\infty}|\{x\in Q: M(f\chi_Q)(x)>\la\}|d\la\\
    	&=|Q| |f|_Q + \int_{|f|_Q}^{\infty}|\{x\in Q: M(f\chi_Q)(x)>\la\}|d\la \\
    	& \geq |Q| |f|_Q + c_1\int_{|f|_Q}^{\infty} \left(\int_{\{x\in Q:
    		|f(x)|>\la\}}|f(y)|\,dy\right)\,\frac{d\la}{\la}\\
    	& =  |Q| |f|_Q + c_1 \int_{\{x\in Q:
    		|f(x)|>|f|_Q\}} \left( \int_{|f|_Q}^{|f(x)|} \frac{d\la}{\la}\right)\,|f(x)|\,dx \\
    	& =  |Q| |f|_Q + c_1 \int_{\{x\in Q:
    		|f(x)|>|f|_Q\}} |f(x)|\,\log \left( \frac{|f(x)|}{|f|_Q}\right)\,dx \\
    	&\geq c_1 \int_{Q}|f|\left(1+\log^+ \frac{|f|}{|f|_{Q}}\right).
    	\end{align*}
    	On the other hand,
    	\begin{align*}
    	\int_{Q}M(f\chi_Q) & = \int_0^{\infty}|\{x\in Q: M(f\chi_Q)(x)>\la\}|d\la \\
    	& \ap \int_0^{\infty}|\{x\in Q: M(f\chi_Q)(x)>2\la\}|d\la \\
    	& = \int_0^{|f|_Q}|\{x\in Q: M(f\chi_Q)(x)>2\la\}|d\la \\
    	&+ \int_{|f|_Q}^{\infty}|\{x\in Q: M(f\chi_Q)(x)>2\la\}|d\la\\
    	& \le |Q| |f|_Q + c_2\int_{|f|_Q}^{\infty} \left(\int_{\{x\in Q:
    		|f(x)|>\la\}}|f(y)|\,dy\right)\,\frac{d\la}{\la}\\
    	& =  |Q| |f|_Q + c_2 \int_{\{x\in Q:
    		|f(x)|>|f|_Q\}} |f(x)|\,\log \left( \frac{|f(x)|}{|f|_Q}\right)\,dx \\
    	&\le c_2 \int_{Q}|f|\left(1+\log^+ \frac{|f|}{|f|_{Q}}\right).
    	\end{align*}
    \end{proof}

    \begin{lem}\label{lem90099009}
        Inequalities
        \begin{equation*}
        \begin{split}
        \frac{1}{|Q|}\int_{Q}|f|\left(1+\log^+ \frac{|f|}{|f|_{Q}}\right)
        \ap \|f\|_{L(1 + \log^+ L),Q}
        \end{split}
        \end{equation*}
        hold for all $f\in \Lloc$ with positive constants independent of $f$
        and $Q$.
    \end{lem}
    \begin{proof}
        Since
        $$
        1\leq \frac{1}{|Q|}\int_{Q}\frac{|f|}{|f|_{Q}}\left(1+\log^+
        \frac{|f|}{|f|_{Q}}\right),
        $$
        then
        \begin{equation*}
        |f|_{Q} \leq \|f\|_{L(1 + \log^+ L),Q}.
        \end{equation*}
        Using the inequality $\log^+ (ab)\leq \log^+ a +\log^+ b$, $a,b
        \in\R^+$, we get
        \begin{align*}
        \frac{1}{|Q|}\int_{Q}|f|\left(1+\log^+
        \frac{|f|}{|f|_{Q}}\right) &\\
        &\hspace{-2cm}=\frac{1}{|Q|}\int_{Q}|f|\left(1+\log^+\left(
        \frac{|f|}{\|f\|_{L(1 + \log^+ L),Q}}\frac{\|f\|_{L(1 + \log^+ L),Q}}{|f|_{Q}}\right)\right) \\
        &\hspace{-2cm}\leq \frac{1}{|Q|}\int_{Q}|f|\left(1+\log^+
        \frac{|f|}{\|f\|_{L(1 + \log^+ L),Q}}\right)
        \\
        &\hspace{-2cm}+\frac{1}{|Q|}\int_{Q}|f|\log^+\frac{\|f\|_{L(1 + \log^+ L),Q}}{|f|_{Q}}\\
        &\hspace{-2cm}\leq \|f\|_{L(1 + \log^+
            L),Q}+|f|_{Q}\log^+\frac{\|f\|_{L(1 + \log^+ L),Q}}{|f|_{Q}}.
        \end{align*}
        Since $\frac{\|f\|_{L(1 + \log^+ L),Q}}{|f|_{Q}}\geq 1$ and $\log
        t\leq t$ when $t\geq 1$, we get
        \begin{equation*}
        \begin{split}
        \frac{1}{|Q|}\int_{Q}|f|\left(1+\log^+ \frac{|f|}{|f|_{Q}}\right)
        \leq 2\|f\|_{L(1 + \log^+ L),Q}.
        \end{split}
        \end{equation*}
        On the other hand, since
        \begin{equation*}
        \|f\|_{L(1 + \log^+ L),Q} =   \frac{1}{|Q|}\int_{Q} |f|\left(1+\log^+ \frac{|f|}{\|f\|_{L(1 + \log^+ L),Q}}\right),
        \end{equation*}
        on using $|f|_{Q} \leq \|f\|_{L(1 + \log^+ L),Q}$, we get that
        $$
        \|f\|_{L(1 + \log^+ L),Q} \le \frac{1}{|Q|}\int_{Q}|f|\left(1+\log^+
        \frac{|f|}{|f|_{Q}}\right).
        $$
    \end{proof}
    
\noindent{\bf \textit {Proof of Theorem \ref{thm6.3}.}} By Lemma \ref{lem6.1}, we get that
$$
\sup_{Q \ni x}|Q|^{\frac{\la-n}{n}}\int_{Q}Mf(y)dy\lesssim
\sup_{Q\ni x}|Q|^{\frac{\la}{n}} \sup_{Q \subset Q'}\|f\|_{L(1+\log^+ L),Q'} \le \sup_{Q \ni x}|Q|^{\frac{\la}{n}} \|f\|_{L(1+\log^+ L),Q}.
$$
The equivalence
$$
\sup_{Q \ni x}|Q|^{\frac{\la}{n}}\|f\|_{L(1+\log^+ L),Q} \thickapprox
\sup_{Q \ni x}|Q|^{\frac{\la-n}{n}}\int_{Q}|f|\left(1+\log^+
\frac{|f|}{|f|_{Q}}\right)
$$
is obvious in view of Lemma \ref{lem90099009}.

By Lemma \ref{stein}, we have that
$$
\sup_{Q \ni x}|Q|^{\frac{\la-n}{n}}\int_{Q}|f|\left(1+\log^+
\frac{|f|}{|f|_{Q}}\right)\lesssim
\sup_{Q \ni x}|Q|^{\frac{\la-n}{n}}\int_{Q}Mf.
$$

$$\hspace{17.5cm}\square$$

    The following corollaries follow from Theorem \ref{thm6.3}.
    \begin{cor}
        Inequalities
        \begin{equation}\label{eq2340958671897252}
        M^2f(x) \thickapprox M_{L(1 + \log^+ L)}f(x) \thickapprox \sup_{x\in
            Q}\frac{1}{|Q|}\int_{Q}|f|\left(1+\log^+ \frac{|f|}{|f|_{Q}}\right)
        \end{equation}
        holds for all $x\in\rn$ and $f\in \Lloc$ with positive constants independent of $x$
        and $f$.
    \end{cor}

    \begin{cor}\label{thm134321r19}
        Let $0<\la <n$. The equivalency
        $$
        \|Mf\|_{\mathcal{M}_{1,\la}} \thickapprox
        \|f\|_{\mathcal{M}_{L(1 + \log^+ L),\la}} \thickapprox  \sup_{Q}|Q|^{\frac{\la-n}{n}}\int_{Q}|f|\left(1+\log^+ \frac{|f|}{|f|_{Q}}\right)
        $$
        holds with positive constants independent of $f$.
    \end{cor}

    Note that $M^2f \thickapprox M_{L(1 + \log^+ L)}f$ was proved in
    \cite{CPer} (see, also \cite[p. 159]{graf}). For the second part of
    \eqref{eq2340958671897252} see \cite{CarPas}, \cite{Leck},
    \cite{LeckN} and \cite{CPer}. The equivalence
    $\|Mf\|_{\mathcal{M}_{1,\la}} \thickapprox \|f\|_{\mathcal{M}_{L(1 + \log^+ L),\la}}$ is a special case of \cite[Lemma 3.5]{sst}.


    \section{Note on the boundedness of the maximal function on Zygmund-Morrey spaces}\label{sect7.3}

    In this section we prove that the Hardy-Littlewood maximal operator $M$ is bounded on $\M_{L(1 + \log^+ L),\la}$, $0 < \la < n$, for radially decreasing functions, and we give an example which shows that $M$ is not bounded on $\M_{L(1 + \log^+ L),\la}$, $0 < \la < n$.

    In order to prove the main result of this section we need the following auxiliary lemmas.
    \begin{lem}\label{lem999.1}
        Assume that $0 < \la <  n$. Let  $f \in \mf^{\rad,\dn}(\rn)$ with $f(x) = \vp (|x|)$. The equivalency
        $$
        \|f\|_{\M_{L(1 + \log^+ L),\la}} \ap \sup_{x>0} x^{\la - n} \int_0^x
        \frac{1}{t} \int_0^t
        |\vp (\rho)|\rho^{n-1}\,d\rho\,dt
        $$
        holds with positive constants independent of $f$.
    \end{lem}
    \begin{proof}
        Recall that
        $$
        \|f\|_{\M_{L(1 + \log^+ L),\la}} \ap
        \sup_{B}|B|^{\frac{\la-n}{n}}\int_{B}M f = \| M_{\la}(Mf)\|_{\infty}, ~ f \in \mf(\rn).
        $$
        Since $M_{\la}(f)(y) \gs \frac{1}{|B(0,|y|)|^{1 - \la /n}}\int_{B(0,|y|)} |f(z)|\,dz$, in view of $Mf \ap Hf$, $f \in
        \mf^{\rad,\dn}$, switching to polar coordinates, we have that
        \begin{align*}
        M_{\la}(Mf)(y) & \gs \frac{1}{|B(0,|y|)|^{1 - \la /n}}\int_{B(0,|y|)} |Mf(z)|\,dz \\
        & \ap \frac{1}{|B(0,|y|)|^{1 - \la /n}}\int_{B(0,|y|)} |Hf(z)|\,dz
        \\
        & = \frac{1}{|B(0,|y|)|^{1 - \la /n}}\int_{B(0,|y|)}
        \frac{1}{|B(0,|z|)} \int_{B(0,|z|)} |f(w)|\,dw\,dz \\
        & \ap \frac{1}{|B(0,|y|)|^{1 -\la /n}}\int_{B(0,|y|)}
        |z|^{-n} \int_0^{|z|} |\vp(\rho)|\rho^{n-1}\,d\rho\,dz \\
        & \ap |y|^{\la - n}\int_0^{|y|} \frac{1}{t} \int_0^t
        |\vp(\rho)|\rho^{n-1}\,d\rho\,dt.
        \end{align*}
        Consequently,
        \begin{align*}
        \|f\|_{\M_{L(1 + \log^+ L),\la}} & \gs \esup_{y \in \rn} |y|^{\la - n}\int_0^{|y|} \frac{1}{t} \int_0^t
        |\vp(\rho)|\rho^{n-1}\,d\rho\,dt  \\
        & = \sup_{x > 0} x^{\la - n}\int_0^x \frac{1}{t} \int_0^t
        |\vp(\rho)|\rho^{n-1}\,d\rho\,dt,
        \end{align*}
        where $f (\cdot) = \vp (|\cdot|)$.

        On the other hand,
        \begin{align*}
        \|f\|_{\M_{L(1 + \log^+ L),\la}} & \ls \sup_{B}|B|^{\frac{\la-n}{n}}\int_0^{|B|}(M f)^* (t)\,dt \\
        & \ap \sup_{B}|B|^{\frac{\la-n}{n}}\int_0^{|B|}f^{**}(t)\,dt \\
        & = \sup_{B}|B|^{\frac{\la-n}{n}}\int_0^{|B|} \frac{1}{t}\int_0^t f^{*}(s)\,ds\,dt \\
        & = \sup_{B}|B|^{\frac{\la-n}{n}}\int_0^{|B|} \frac{1}{t}\int_0^t |\vp(s^{\frac{1}{n}})|\,ds\,dt \\
        & \ap \sup_{B}|B|^{\frac{\la-n}{n}}\int_0^{|B|} \frac{1}{t}\int_0^{t^{\frac{1}{n}}} |\vp(\rho)|\rho^{n-1}\,d\rho\,dt \\
        & \ap \sup_{B}|B|^{\frac{\la-n}{n}}\int_0^{|B|^{\frac{1}{n}}} \frac{1}{x}\int_0^{x} |\vp(\rho)|\rho^{n-1}\,d\rho\,dx \\
        & = \sup_{x > 0} x^{\la - n}\int_0^x \frac{1}{t} \int_0^t
        |\vp(\rho)|\rho^{n-1}\,d\rho\,dt,
        \end{align*}
        where  $f (\cdot) = \vp (|\cdot|)$.
    \end{proof}

    \begin{cor}\label{cor999.1}
        Assume that $0 < \la <  n$. Let  $f \in \mf^{\rad,\dn}(\rn)$ with $f(x) = \vp (|x|)$. The equivalency
        $$
        \|Mf\|_{\M_{L(1 + \log^+ L),\la}} \ap \sup_{x>0} x^{\la - n} \int_0^x \frac{1}{y} \int_0^y
        \frac{1}{t} \int_0^t
        \vp (\rho)\rho^{n-1}\,d\rho\,dt\,dy
        $$
        holds with positive constants independent of $f$.
    \end{cor}
    \begin{proof}
        Let $f \in \mf^{\rad,\dn}$  with $f(x) = \vp (|x|)$. Since $Mf \ap Hf$ and $Hf \in \mf^{\rad,\dn}$, by Lemma \ref{lem999.1}, switching to polar coordinates, we have that
        \begin{align*}
        \|Mf\|_{\M_{L(1 + \log^+ L),\la}} & \ap \sup_{x>0} x^{\la - n} \int_0^x
        \frac{1}{y} \int_0^y
        \left(\frac{1}{|B(0,t)|} \int_{B(0,t)} |f(y)|\,dy \right)t^{n-1}\,dt\,dy \\
        & \ap \sup_{x>0} x^{\la - n} \int_0^x \frac{1}{y} \int_0^y
        \frac{1}{t} \int_0^t
        \vp (\rho)\rho^{n-1}\,d\rho\,dt\,dy.
        \end{align*}
    \end{proof}

    \begin{lem}\label{lem999.2}
        Assume that $0 < \la <  n$. Let  $f \in \mf^{\rad,\dn}$ with $f(x) = \vp (|x|)$. The inequality
        $$
        \|Mf\|_{\M_{L(1 + \log^+ L),\la}} \ls \|f\|_{\M_{L(1 + \log^+
                L),\la}},~ f \in \mf^{\rad,\dn}
        $$
        holds if and only if the inequality
        \begin{align*}
        \sup_{x>0} x^{\la - n} \int_0^x \frac{1}{y} \int_0^y
        \frac{1}{t} \int_0^t
        \vp (\rho)\rho^{n-1}\,d\rho\,dt\,dy & \\
        & \hspace{-3cm} \ls \sup_{x>0} x^{\la - n} \int_0^x
        \frac{1}{t} \int_0^t
        \vp (\rho)\rho^{n-1}\,d\rho\,dt,~ \vp \in \mf^{+,\dn}(\R_+)
        \end{align*}
        holds true.
    \end{lem}
    \begin{proof}
        The statement immediately follows from Lemma \ref{lem999.1} and Corollary \ref{cor999.1}.
    \end{proof}

    \begin{lem}\label{lem999.5}
        Let $0 <\la < n$. Then inequality
        \begin{equation}\label{eq.999}
        \sup_{x>0} x^{\la - n} \int_0^x \frac{1}{y} \int_0^y
        \frac{1}{t} \int_0^t
        \vp (\rho)\rho^{n-1}\,d\rho\,dt\,dy \ls \sup_{x>0} x^{\la - n} \int_0^x
        \frac{1}{t} \int_0^t
        \vp (\rho)\rho^{n-1}\,d\rho\,dt
        \end{equation}
        holds for all $\vp \in \mf^{+,\dn}(\R_+)$.
    \end{lem}
    \begin{proof}
        Indeed:
        \begin{align*}
        \sup_{x>0} x^{\la - n} \int_0^x \frac{1}{y} \int_0^y
        \frac{1}{t} \int_0^t
        \vp (\rho)\rho^{n-1}\,d\rho\,dt\,dy & \\
        & \hspace{-5cm} =
        \sup_{x>0} x^{\la - n} \int_0^x y^{n - \la - 1} y^{\la - n} \int_0^y
        \frac{1}{t} \int_0^t
        \vp (\rho)\rho^{n-1}\,d\rho\,dt\,dy \\
        & \hspace{-5cm} \le \sup_{y > 0}  y^{\la - n} \int_0^y \frac{1}{t} \int_0^t
        \vp (\rho)\rho^{n-1}\,d\rho\,dt \cdot \left( \sup_{x > 0} x^{\la - n} \int_0^x y^{n - \la - 1} dy \right) \\
        &  \hspace{-5cm} \ap \sup_{y > 0}  y^{\la - n} \int_0^y \frac{1}{t} \int_0^t
        \vp (\rho)\rho^{n-1}\,d\rho\,dt.
        \end{align*}
    \end{proof}

    \begin{thm}\label{main11}
        Assume that $0 < \la <  n$. The inequality
        $$
        \|Mf\|_{\M_{L(1 + \log^+ L),\la}} \ls \|f\|_{\M_{L(1 + \log^+ L),\la}}
        $$
        holds for all $f \in \mf^{\rad,\dn}$ with constant independent of $f$.
    \end{thm}
    \begin{proof}
        The statement follows by Lemmas \ref{lem999.2} and \ref{lem999.5}.
    \end{proof}

    \begin{ex}
        We give an example which shows that $M$ is not bounded on $\M_{L(1 +
            \log^+ L),\la}$, $0 < \la < n$. For simplicity let $n = 1$ and $\la
        = 1 / 2$. Consider even function $f$ defined as follows:
        $$
        f(x) = \sum_{k=0}^{\infty} \chi_{[k^2 \ln^2 (k+e),k^2 \ln^2 (k+e) +
            1]}(x), \qquad x \ge 0.
        $$
        It is easy to see that $Mf$ and $M^2 f$ are even functions. Obviously,
        \begin{align*}
        Mf(x) \ap & \sum_{k = 0}^{\infty} \chi_{[k^2 \ln^2 (k + e),k^2 \ln^2 (k + e) + 1]}(x) \\
        & + \sum_{k = 0}^{\infty} \frac{1}{x - k^2 \ln^2 (k + e)}\chi_{[k^2 \ln^2 (k + e) + 1, k^2 \ln^2 (k + e) + 1 + m_k]}(x) \\
        & + \sum_{k = 0}^{\infty} \frac{1}{(k+1)^2\ln^2 (k + 1 + e) + 1 -
            x}\chi_{[k^2 \ln^2 (k + e) + 1 + m_k, (k+1)^2 \ln^2 (k + 1 +
            e)]}(x), ~ x \ge 0,
        \end{align*}
        where
        $$
        m_k = \frac{(k+1)^2 \ln^2 (k + 1 + e) - k^2 \ln^2 (k + e) - 1}{2}, ~ k = 0,1,2,\ldots.
        $$
        Then
        \begin{align*}
        \|f\|_{\M_{L(1 + \log^+ L),1 / 2}(\R)} & \ap \|Mf\|_{{\mathcal M}_{1,1/2}(\R)} = \sup_{I}|I|^{-1/2} \int_I Mf \\
        & \le \sup_{I:\,|I| \le 1}|I|^{-1/2} \int_I Mf + \sup_{I:\,|I| > 1}|I|^{-1/2} \int_I Mf.
        \end{align*}
        It is easy to see that
        $$
        \sup_{I:\,|I| \le 1}|I|^{-1/2} \int_I Mf \le \sup_{I:\,|I| \le 1}|I|^{1/2} \le 1.
        $$
        Since
        $$
        \int_{j^2 \ln^2 (j+e)}^{(j+1)^2 \ln^2 (j+e+1)} Mf (x)\,dx \ap (1 + 2 \ln (1
        + m_j)), ~ j = 0,1,2,\ldots,
        $$
        we have that
        \begin{align*}
        \sup_{I:\,|I| > 1}|I|^{-1/2} \int_I Mf (x)\,dx & = \sup_{m \ge 2} \sup_{I:\,m - 1 < |I| \le m}|I|^{-1/2} \int_I Mf (x)\,dx \\
        & \ls \sup_{m \ge 2} m^{-1/2} \int_0^m  Mf (x)\,dx \\
        & \le \sup_{m \ge 2} m^{-1/2} \sum_{i^2 \ln^2 (j+e) < m}\int_{j^2
            \ln^2 (j+e)}^{(j+1)^2 \ln^2 (j+e+1)} Mf (x)\,dx \\
        & \ap \sup_{m \ge 2} m^{-1/2} \sum_{i^2 \ln^2 (j+e) < m} (1 + 2 \ln (1
        + m_j)) \\
        & \ls \sup_{m \ge 2} m^{-1/2} \sum_{i^2 \ln^2 (j+e) < m} \ln (j + e)
        \\
        & \ls \sup_{m \ge 2} m^{-1/2} m^{1/2} = 1,
        \end{align*}
        we have that
        $$
        \|f\|_{\M_{L(1 + \log^+ L),1 / 2}(\R)} \ls 1 + 1 = 2.
        $$

        On the other hand, it is easy to see that \begin{align*}
        M^2 f(x) & \ge \frac{1}{x - (k^2 \ln^2 (k+e) + 1)} \int_{k^2 \ln^2 (k + e) + 1}^x \frac{dt}{t - k^2 \ln^2 (k + e)} \\
        & = \frac{\ln (x - k^2 \ln^2 (k+e))}{x - (k^2 \ln^2 (k+e) + 1)} \\
        & \ge \frac{\ln (x - k^2 \ln^2 (k+e))}{x - k^2 \ln^2 (k+e)}
        \end{align*}
        for any $x \in [k^2 \ln^2 (k + e) + e,k^2 \ln^2 (k + e) + m_k]$.

        Thus
        \begin{align*}
        M^2 f(x) & \ge \sum_{k = 0}^{\infty} \frac{\ln (x - k^2 \ln^2 (k+e))}{x - k^2 \ln^2 (k+e)}\chi_{[k^2 \ln^2 (k + e) + e,k^2 \ln^2 (k +
            e) + m_k]}(x).
        \end{align*}
        Finally,
        \begin{align*}
        \|M f\|_{\M_{L(1 + \log^+ L),1 / 2}(\R)} & \ap \|M^2 f\|_{{\mathcal M}_{1,1/2}(\R)} \\
        & \gs \sup_k (k \ln (k + e))^{-1} \int_0^{k^2 \ln^2 (k + e)} M^2 f (x)\,dx \\
        & \ge \sup_k (k \ln (k + e))^{-1} \sum_{j=1}^{k-1}\int_{j^2 \ln^2 (j + e) + e}^{j^2 \ln^2 (j + e) + m_j} M^2 f (x)\,dx\\
        & \ge \sup_k (k \ln (k + e))^{-1} \sum_{j=1}^{k-1}\int_{j^2 \ln^2 (j + e) + e}^{j^2 \ln^2 (j + e) + m_j} \frac{\ln (x - k^2 \ln^2 (k+e))}{x - k^2 \ln^2 (k+e)} \,dx \\
        & = \sup_k (k \ln (k + e))^{-1} \sum_{j=1}^{k-1}\int_e^{m_j} \frac{\ln x}{x}\,dx \\
        & \gs \sup_k (k\ln(k + e))^{-1} \sum_{j=1}^{k-1} \ln^2 m_j \\
        & \gs \sup_k (k\ln(k + e))^{-1} \sum_{j=1}^{k-1} \ln^2 (j + e) \\
        & \gs \sup_k (k\ln(k + e))^{-1} k \ln^2 (k + e) \\
        & = \sup_k \ln (k + e) = \infty.
        \end{align*}
    \end{ex}


    \section{Weak-type estimates in Morrey spaces for the iterated maximal function}\label{sect7.5}

    In this section the boundedness of the iterated maximal operator
    $M^2$ from Zygmund-Morrey spaces $\mathcal{M}_{L(1 + \log^+ L),\la}$
    to weak Zygmund-Morrey spaces $\WM_{L(1 + \log^+ L),\la}$ is proved.

    \begin{thm}\label{thm53432}
        Let $0<\la <n$. Then the operator $M^2$ is bounded from
        $\mathcal{M}_{L(1 + \log^+ L),\la}$ to $\WM_{L(1 + \log^+ L),\la}$
        and the following inequality holds
        \begin{equation}\label{eq88008802}
        \begin{split}
        \|M^2f\|_{\WM_{L(1 + \log^+ L),\la}}\leq c \|f\|_{\mathcal{M}_{L(1 +
                \log^+ L),\la}}
        \end{split}
        \end{equation}
        with positive constant $c$ independent of $f$.
    \end{thm}
    \begin{proof}
        Let $Q$ be any cube in $\rn$
        and let $f=f_1+f_2$, where $f_1=f\chi_{4Q}$. By subadditivity of
        $M^2$ we get
        \begin{equation*}
        M^2f \leq M^2f_1 +M^2f_2.
        \end{equation*}
        Since for any cube $Q'$ conditions $z\in 2Q\cap Q'$ and
        $Q'\cap\{\rn\backslash 4Q\}\neq\emptyset$ imply $Q\subset 4Q'$, we
        have
        \begin{equation}\label{eq998855}
        Mf_2(z)=M(f\chi_{\rn\backslash 4Q})(z)\leq \sup_{Q\subset
            4Q'}\frac{1}{|Q'|}\int_{Q'}|f|
        \end{equation}
        for any $z\in 2Q$. Thus for any $z\in \rn$
        \begin{equation}\label{eq1122223333}
        Mf_2(z)\leq \chi_{2Q}(z)\sup_{Q\subset
            4Q'}\frac{1}{|Q'|}\int_{Q'}|f| +\chi_{\rn\backslash 2Q}(z)Mf(z).
        \end{equation}
        Applying to both sides of the inequality \eqref{eq1122223333} by
        operator $M$ for any $y\in Q$ we get
        \begin{equation}
        M^2f_2(y)\leq M(\chi_{2Q})(y)\sup_{Q\subset
            4Q'}\frac{1}{|Q'|}\int_{Q'}|f| +M(\chi_{\rn\backslash 2Q}Mf)(y).
        \end{equation}
        Since $M(\chi_{2Q})(y)=1$, $y\in Q$, by the inequality
        \eqref{eq998855} we arrive at
        \begin{equation}
        M^2f_2(y)\leq \sup_{Q\subset 4Q'}\frac{1}{|Q'|}\int_{Q'}|f|
        +\sup_{Q\subset 2Q'}\frac{1}{|Q'|}\int_{Q'}Mf \lesssim
        \sup_{Q\subset Q'}\frac{1}{|Q'|}\int_{Q'}Mf.
        \end{equation}
        Consequently, for $y\in Q$
        \begin{equation}\label{eq5.15}
        M^2f(y)\lesssim M^2(f\chi_{4Q})(y) + \sup_{Q\subset
            Q'}\frac{1}{|Q'|}\int_{Q'}Mf.
        \end{equation}
        In view of inequality 
       \begin{equation}\label{log}
       1+\log^+ (ab)\leq (1+\log^+ a)(1+\log^+ b),
       \end{equation}        
        by Lemma \ref{lem002.8}, for any $\a
        >0$ and $t>0$ we have that
        \begin{equation*}
        \begin{split}
        \left|\left\{x\in Q: M^2(f\chi_{4Q})(x)>\a t
        \right\}\right| \\
        & \hspace{-2cm}\leq \left|\left\{x\in \rn:
        M^2(f\chi_{4Q})(x)>\a t \right\}\right| \\
        & \hspace{-2cm} \leq c \int_{\rn}\frac{|(f\chi_{4Q})(x)|}{\a
            t}\left(1+\log^+
        \left(\frac{|(f\chi_{4Q})(x)|}{\a t}\right)\right)dx \\
        & \hspace{-2cm} \leq c \frac{1}{\a}\left(1+\log^+
        \frac{1}{\a}\right) \int_{4Q}\frac{|f(x)|}{t}\left(1+\log^+
        \left(\frac{|f(x)|}{t}\right)\right)dx .
        \end{split}
        \end{equation*}
        We get that
        \begin{equation*}
        \begin{split}
        \frac{\left|\left\{x\in Q: M^2(f\chi_{4Q})(x)>\a t
            \right\}\right|}{\frac{1}{\a}\left(1+\log^+ \frac{1}{\a}\right)}
        \leq c  \int_{4Q}\frac{|f(x)|}{t}\left(1+\log^+
        \left(\frac{|f(x)|}{t}\right)\right)dx .
        \end{split}
        \end{equation*}
        Consequently,
        \begin{equation*}
        \begin{split}
        \sup_{\a >0}\frac{1}{|Q|}\frac{\left|\left\{x\in Q:
            M^2(f\chi_{4Q})(x)>\a t \right\}\right|}{\frac{1}{\a}\left(1+\log^+
            \frac{1}{\a}\right)} \leq c \, \frac{1}{|4Q|}
        \int_{4Q}\frac{|f(x)|}{t}\left(1+\log^+
        \left(\frac{|f(x)|}{t}\right)\right)dx .
        \end{split}
        \end{equation*}
        Thus
        \begin{equation*}
        \begin{split}
        \inf\left\{t>0 : \sup_{\a
            >0}\frac{1}{|Q|}\frac{\left|\left\{x\in Q: M^2(f\chi_{4Q})(x)>\a
            t \right\}\right|}{\frac{1}{\a}\left(1+\log^+ \frac{1}{\a}\right)}
        \leq 1 \right\} &
        \\
        &\hspace{-5cm}\leq \inf\left\{t>0 :  \frac{1}{|4Q|}
        \int_{4Q}\frac{c|f(x)|}{t}\left(1+\log^+
        \left(\frac{|f(x)|}{t}\right)\right)dx\leq 1 \right\}\\
        \\
        &\hspace{-5cm}\leq \inf\left\{t>0 :  \frac{1}{|4Q|}
        \int_{4Q}\frac{c|f(x)|}{t}\left(1+\log^+
        \left(\frac{c|f(x)|}{t}\right)\right)dx\leq 1 \right\}
        \end{split}
        \end{equation*}
        that is,
        \begin{equation}\label{eq880088}
        \begin{split}
        \|M^2(f\chi_{4Q})\|_{WL(1 + \log^+ L),Q}\leq \|c f\|_{L(1 + \log^+
            L),4Q}=c\|f\|_{L(1 + \log^+ L),4Q}.
        \end{split}
        \end{equation}
        For the second summand in right hand side of the inequality
        \eqref{eq5.15} applying the inequality \eqref{eq44005500} we obtain
        \begin{equation}\label{eq5784756}
        \begin{split}
        \left\|\sup_{Q\subset Q'}\frac{1}{|Q'|}\int_{Q'}Mf
        \right\|_{WL(1 + \log^+ L),Q}  \lesssim \sup_{Q\subset Q'}\frac{1}{|Q'|}\int_{Q'}Mf
        \lesssim \sup_{Q\subset Q'}\|f\|_{L(1 + \log^+ L),Q'}.
        \end{split}
        \end{equation}
        By inequalities \eqref{eq5.15}, \eqref{eq880088} and
        \eqref{eq5784756} we get
        \begin{equation}\label{eq88008801}
        \begin{split}
        \|M^2f\|_{WL(1 + \log^+ L),Q}\leq c \sup_{Q\subset 4Q'}\|f\|_{L(1 +
            \log^+ L),Q'}.
        \end{split}
        \end{equation}
        Thus
        \begin{align*}
        \sup_{Q}|Q|^{\frac{\la}{n}}\|M^2f\|_{WL(1 + \log^+ L),Q} & \leq c \, \sup_{Q}|Q|^{\frac{\la}{n}}\sup_{Q\subset
            4Q'}\|f\|_{L(1 + \log^+ L),Q'}\\
        &\leq c \,
        \left(\sup_{Q}|Q|^{\frac{\la}{n}} \sup_{Q\subset
            4Q'}|Q'|^{-\frac{\la}{n}}\right) \,\sup_{Q}|Q|^{\frac{\la}{n}}\|f\|_{L(1 + \log^+ L),Q}\\
        &\thickapprox \sup_{Q}|Q|^{\frac{\la}{n}}\|f\|_{L(1 +
            \log^+ L),Q},
        \end{align*}
        that is,
        \begin{equation*}
        \begin{split}
        \|M^2f\|_{\WM_{L(1 + \log^+ L),\la}}\leq c \|f\|_{\mathcal{M}_{L(1 +
                \log^+ L),\la}}.
        \end{split}
        \end{equation*}
    \end{proof}


    \section{Weak-type estimates in Morrey spaces for maximal commutator and commutator of maximal function}\label{sect8}

    In this section the class of functions for which the maximal
    commutator $C_b$ is bounded from  $\mathcal{M}_{L(1 + \log^+
        L),\la}$ to $\WM_{L(1 + \log^+ L),\la}$ are characterized. It is
    proved that the commutator of the Hardy-Littlewood maximal operator
    $M$ with function $b \in \B(\rn)$ such that $b^- \in
    L_{\infty}(\rn)$ is bounded from $\mathcal{M}_{L(1 + \log^+ L),\la}$
    to $\WM_{L(1 + \log^+ L),\la}$.

    \begin{thm}\label{main10}
        Let $0<\la <n$. The following assertions are equivalent:

        {\rm (i)} $b\in\B(\rn)$.

        {\rm (ii)} The operator $C_b$ is bounded from $\mathcal{M}_{L(1 +
            \log^+ L),\la}$ to $\WM_{L(1 + \log^+ L),\la}$.
    \end{thm}
    \begin{proof}
        {${\rm (i)}\Rightarrow {\rm (ii)}$}. Assume that $b\in\B(\rn)$. By
        Theorem \ref{lem1111111} and Theorem \ref{thm53432} operator $C_b$
        is bounded from $\mathcal{M}_{L(1 + \log^+ L),\la}$ to
        $\WM_{L(1 + \log^+ L),\la}$ and the following inequality
        holds
        \begin{equation}\label{eq8800887645}
        \begin{split}
        \|C_b(f)\|_{\WM_{L(1 + \log^+ L),\la}}\leq c \|b\|_{*}
        \|f\|_{\mathcal{M}_{L(1 + \log^+ L),\la}}
        \end{split}
        \end{equation}
        with positive constant $c$ independent of $f$.

        {${\rm (ii)}\Rightarrow {\rm (i)}$}. Assume that the inequality
        \begin{equation}\label{eq34544724r2346657u}
        \begin{split}
        \|C_b(f)\|_{\WM_{L(1 + \log^+ L),\la}}\leq c \|f\|_{\mathcal{M}_{L(1
                + \log^+ L),\la}}.
        \end{split}
        \end{equation}
        holds with positive constant $c$ independent of $f$. Let $Q_0$ be
        any cube in $\rn$ and let $f=\chi_{Q_0}$.

        By Theorem \ref{thm6.3},
        \begin{equation*}
        \begin{split}
        \|\chi_{Q_0}\|_{\mathcal{M}_{L(1 + \log^+ L),\la}} & \thickapprox
        \sup_{Q}|Q|^{\frac{\la-n}{n}}\int_{Q}\chi_{Q_0}\left(1+\log^+
        \frac{\chi_{Q_0}}{(\chi_{Q_0})_{Q}}\right)\\
        &=\sup_{Q:\,Q\cap Q_0 \neq \emptyset}|Q|^{\frac{\la}{n}}\frac{|Q\cap Q_0|}{|Q|}\left(1+\log
        \frac{|Q|}{|Q\cap Q_0|}\right).
        \end{split}
        \end{equation*}
        Obviously,
        $$
        \|\chi_{Q_0}\|_{\mathcal{M}_{L(1 + \log^+ L),\la}}  \gs
        |Q_0|^{\frac{\la}{n}}.
        $$
        Let $\ve \in (0,1-\la / n)$. Since the function $(1+\log t)/ t^{\ve}$ is bounded on the interval
        $[1,\infty)$, we get
        \begin{align*}
        \|\chi_{Q_0}\|_{\mathcal{M}_{L(1 + \log^+ L),\la}}  & \ls
        \sup_{Q:\,Q\cap Q_0 \neq \emptyset}|Q|^{\frac{\la}{n}}\frac{|Q\cap
            Q_0|}{|Q|}\left(
        \frac{|Q|}{|Q\cap Q_0|}\right)^{\ve} \\
        & = \sup_{Q:\,Q\cap Q_0 \neq \emptyset}|Q|^{\frac{\la}{n}+\ve-1}|Q\cap Q_0|^{1-\ve} \\
        & = \sup_{Q \subseteq Q_0}|Q|^{\frac{\la}{n}+\ve-1}|Q\cap Q_0|^{1-\ve} =|Q_0|^{\frac{\la}{n}}.
        \end{align*}
        Thus
        \begin{equation}\label{eq56789345678}
        \|\chi_{Q_0}\|_{\mathcal{M}_{L(1 + \log^+ L),\la}}  \thickapprox
        |Q_0|^{\frac{\la}{n}}.
        \end{equation}
        On the other hand
        \begin{equation*}
        \begin{split}
        \|C_b(\chi_{Q_0})\|_{W\mathcal{M}_{L(1 + \log^+ L),\la}}&=\sup_{Q}|Q|^{\frac{\la}{n}}\|C_b(\chi_{Q_0})\|_{WL(1 + \log^+ L),Q}\\
        &\geq |Q_0|^{\frac{\la}{n}}\|C_b(\chi_{Q_0})\|_{WL(1 + \log^+
            L),Q_0}.
        \end{split}
        \end{equation*}
        Note that
        \begin{equation*}
        \begin{split}
        \|C_b(\chi_{Q_0})\|_{WL(1 + \log^+ L),Q_0} &\\
        &\hspace{-2cm}=\inf \left\{\la>0:
        \sup_{t>0}\frac{1}{|Q_0|}\frac{|\{x\in Q_0: |C_b(\chi_{Q_0})(x)|>\la
            t\}|}{\frac{1}{t}\left(1+\log^+ \frac{1}{t}\right)}\leq 1\right\} \\
        &\hspace{-2cm}\geq \inf \left\{\la>0: \frac{2}{|Q_0|}|\{x\in Q_0:
        |C_b(\chi_{Q_0})(x)|> 2\la\}|\leq 1\right\}.
        \end{split}
        \end{equation*}
        Since for any  $x\in Q_0$
        \begin{equation*}
        C_b(\chi_{Q_0})(x)\geq \frac{1}{|Q_0|}\int_{Q_0}|b(x)-b(y)|dy \geq
        \frac{1}{2|Q_0|}\int_{Q_0}|b(y)-b_{Q_0}|dy,
        \end{equation*}
        then
        \begin{equation*}
        \begin{split}
        \frac{2}{|Q_0|}\left|\left\{x\in Q_0: |C_b(\chi_{Q_0})(x)|> 2\,
        \frac{1}{4|Q_0|}\int_{Q_0}|b(y)-b_{Q_0}|dy\right\}\right|=2.
        \end{split}
        \end{equation*}
        Thus
        \begin{equation*}
        \|C_b(\chi_{Q_0})\|_{WL(1 + \log^+ L),Q_0}  \geq
        \frac{1}{4|Q_0|}\int_{Q_0}|b(y)-b_{Q_0}|dy.
        \end{equation*}
        Consequently,
        \begin{equation}\label{eq87678656566}
        \begin{split}
        \|C_b(\chi_{Q_0})\|_{W\mathcal{M}_{L(1 + \log^+ L),\la}}\gtrsim
        |Q_0|^{\frac{\la}{n}}\frac{1}{|Q_0|}\int_{Q_0}|b(y)-b_{Q_0}|dy.
        \end{split}
        \end{equation}
        By \eqref{eq34544724r2346657u}, \eqref{eq56789345678} and
        \eqref{eq87678656566} we arrive at
        $$
        \frac{1}{|Q_0|}\int_{Q_0}|b(y)-b_{Q_0}|dy \ls c.
        $$
    \end{proof}

    Combining Theorems  \ref{lem1111111} and \ref{main11}, we get the following statement.
    \begin{thm}
        Let $0<\la <n$. Assume that $b\in\B(\rn)$.
        Then the operator $C_b$ is bounded on $\mathcal{M}_{L(1 + \log^+ L),\la}$ for radially decreasing functions.
    \end{thm}

    The following theorems hold true.
    \begin{thm}\label{thm4567895432}
        Let $0<\la <n$ and $b$ is in $\B(\rn)$ such that $b^-\in
        L_{\infty}(\rn)$. Then the operator $[M,b]$ is bounded from
        $\mathcal{M}_{L(1 + \log^+ L),\la}$ to $\WM_{L(1 + \log^+ L),\la}$
        and the following inequality holds
        \begin{equation*}
        \|[M,b]f\|_{\WM_{L(1 + \log^+ L),\la}}\leq
        c\left(\|b^+\|_{*}+\|b^-\|_{L_{\infty}}\right)
        \|f\|_{\mathcal{M}_{L(1 + \log^+ L),\la}}
        \end{equation*}
        with positive constant $c$ independent of $f$.
    \end{thm}
    \begin{proof}
        The statement follows by Theorem \ref{lem1111111} and Theorem
        \ref{thm53432}.
    \end{proof}

    \begin{thm}\label{thm456789543223}
        Let $0<\la <n$ and $b$ is in $\B(\rn)$ such that $b^-\in
        L_{\infty}(\rn)$. Then the operator $[M,b]$ is bounded on
        $\mathcal{M}_{L(1 + \log^+ L),\la}$ for radially decreasing functions,
        and the following inequality holds
        \begin{equation*}
        \|[M,b]f\|_{\mathcal{M}_{L(1 + \log^+ L),\la}}\leq
        c\left(\|b^+\|_{*}+\|b^-\|_{L_{\infty}}\right)
        \|f\|_{\mathcal{M}_{L(1 + \log^+ L),\la}},~ f \in \mf^{\rad,\dn},
        \end{equation*}
        with positive constant $c$ independent of $f$.
    \end{thm}
    \begin{proof}
        The statement follows by Theorems \ref{lem1111111} and \ref{main11}.
    \end{proof}

\end{document}